\documentclass[12pt,a4paper]{amsart}
\title[Hopf-Galois structures on Quaternionic extensions]{The structure of Hopf algebras giving Hopf-Galois structures on Quaternionic extensions}
\author{Stuart Taylor, Paul J. Truman}
\date{\today}
\usepackage{amsmath}
\usepackage{amsfonts}
\usepackage{amssymb}
\usepackage{amsthm}
\usepackage{cleveref}
\Crefname{table}{Table}{Tables}
\usepackage{tikz-cd}
\usepackage[utf8]{inputenc}

\newcommand{\perm}{\text{Perm}}
\newcommand{\stab}{\text{Stab}}

\newcommand{\End}{\text{End}}

\newcommand{\gal}{\text{Gal}}
\newcommand{\Q}{\mathbb{Q}}
\newenvironment{myitemize}{\begin{itemize}\setlength{\itemsep}{0pt}\setlength{\parskip}{0pt}\setlength{\parsep}{0pt}}{\end{itemize}}

\newtheorem{theorem}{Theorem}[section]
\newtheorem{lemma}[theorem]{Lemma}

\newtheorem{cor}[theorem]{Corollary}

\newtheorem{remark}[theorem]{Remark}
\newtheorem{example}[theorem]{Example}

\begin{document}

\maketitle

\begin{abstract}
Let $L/F$ be a Galois extension of fields with Galois group isomorphic to the quaternion group of order $ 8 $. We describe all of the Hopf-Galois structures admitted by $ L/F $, and determine which of the Hopf algebras that appear are isomorphic as Hopf algebras. In the case that $ F $ has characteristic zero we also determine which of these Hopf algebras are isomorphic as $ F $-algebras and explicitly compute their Wedderburn-Artin decompositions. 
\end{abstract}

\section{Introduction}

Let $L/F$ be a finite Galois extension of fields with group $G$. The group algebra $ F[G] $, with its usual action on $ L $, is an example of a {\em Hopf-Galois structure} on the extension. If $H$ is a finite dimensional $F$-Hopf algebra, then we say that $H$ gives a Hopf Galois structure on $L/F$ if and only if the following conditions hold:
\begin{myitemize}
\item $L$ is an $H$-module algebra; that is: $L$ is an $H$-module with action $h(x)$ for $h\in H$ and $x\in L$ where, for $y\in L$, $h(xy)=\sum_{(h)}h_{(1)}(x)h_{(2)}(y)$ (Sweedler notation) and $h(1)=\epsilon(h)(1)$;\\
\item the $F$-linear map $j:L\otimes_F H\to\End_K(L)$ given by $j(l\otimes h)(x)=lh(x)$ for $l,x\in L$, $h\in H$, is bijective.
\end{myitemize}

We note that in this definition $ L/F $ may be taken to be merely an extension of commutative rings. However, in this paper we will be concerned exclusively with fields, specifically the case where $L/F$ is Galois (in the usual sense). 

Since a Hopf-Galois structure on an extension $ L/F $ consists of a Hopf algebra $ H $ and an action of $ H $ on $ L $, it is possible for distinct Hopf-Galois structures on $ L/F $ to involve Hopf algebras that are isomorphic, either as $ F $-Hopf algebras or as $ F $-algebras. These phenomena have recently been studied in papers such as \cite{TARP2017} and \cite{TARP2017_2}. In particular, \cite{TARP2017_2} studies in detail the Hopf-Galois structures admitted by a dihedral extension of fields of degree $ 2p $, where $ p $ is an odd prime. In this paper we perform a similar analysis of the Hopf-Galois structures admitted by a Galois extension of fields with Galois group isomorphic to $ Q_{8} $, the quaternion group of order $ 8 $. We call such extensions \emph{quaternionic}. In addition to continuing and complementing the work begun in the papers cited above, our results have applications in the study of the Hopf-Galois module structure of rings of algebraic integers in quaternionic extensions of local or global fields. Since such extensions have been important in the history of Galois module structure (see \cite{MAR1971}, for example), this has the potential to be a fruitful line of inquiry, which we intend to pursue in a future paper. 

%If $ L/F $ is an extension of local or global fields then the Hopf-Galois structures admitted by $ L/F $ have applications to questions of Galois module structure. Let $\mathfrak{O}_L$, $\mathfrak{O}_F$ be the rings of integers for $L$, $F$ respectively. Classically we are interested in studying the structure of $\mathfrak{O}_L$ as a module over its associated order in $ F[G] $:
%\[ \mathfrak{A}_{F[G]} = \{ z \in F[G] \,|\, z \mathfrak{O}_{L} \subset \mathfrak{O}_{L} \}. \]
%If $ H $ is some other Hopf algebra giving a Hopf-Galois structure on the extension then we may define $ \mathfrak{A}_{H} $ analogously to $ \mathfrak{A}_{F[G]} $, study the structure of $ \mathfrak{O}_{L} $ as a module over $ {\mathfrak A}_{H} $. There are examples of Galois extensions $ L/F $ of $ p $-adic fields for which $ \mathfrak{O}_{L} $ is not free over $ \mathfrak{A}_{F[G]} $ but is free over its associated order in some other Hopf algebra giving a Hopf-Galois structure on the extension \cite{BYOTT1997}.  

A theorem of Greither and Pareigis (\cite[Theorem 3.1]{GRE1987}, see also  \cite[Theorem 6.8]{CHI2000}) classifies all of the Hopf-Galois structures admitted by a finite separable extension of fields. We state it here in a weakened form applicable to finite Galois extensions. Consider the group of permutations on the underlying set of $G$, $\perm(G)$, and let $ \lambda : G \hookrightarrow \perm(G) $ be the left regular representation. A subgroup $N$ of $\perm(G)$ is said to be \emph{regular} if $|N|=|G|$, the stabiliser $\stab_N(g)=\{\eta\in N\,|\,\eta\cdot g=g\}$ is trivial for all $g\in G$, and $N$ acts transitively on $G$ (any two of these properties imply the third). The theorem of Greither and Pareigis states that there is a bijection between regular subgroups $N$ of $\perm(G)$ normalised by $\lambda(G)$ and Hopf Galois structures on $L/F$. Furthermore, if $ N $ is a regular subgroup of $ \perm(G) $ normalised by $\lambda(G) $ then the Hopf algebra giving the Hopf-Galois structure corresponding to $ N $ is $ L[N]^{G} $, the fixed ring of the group algebra $ L[N] $, where $ G $ acts on $ L[N] $ by acting on $ L $ as Galois automorphisms and on $ N $ by $ \,^{g} \eta = \lambda(g)n\lambda(g^{-1})$ for all $ \eta  \in N $, $ g \in G $. For a Hopf algebra $H=L[N]^G$ giving a Hopf Galois structure on $L/F$, we refer to $N$ as the \emph{underlying group} of $H$ and its isomorphism class as the \emph{type} of $H$, or the structure given by $ H $.

\begin{example}
Let $ \rho : G \hookrightarrow \perm(G) $ be the right regular representation. Suppose $g,h\in G$ and $x\in G$. Then $\lambda(g)\rho(h)[x]=gxh^{-1}=\rho(h)[gx]=\rho(h)\lambda(g)[x]$. That is: $\lambda(g)\rho(h)=\rho(h)\lambda(g)$ for all $g,h\in G$. Thus the action of $G$ on $\rho(G)$ is trivial, and so the Hopf algebra $L[\rho(G)]^G$ is in fact the group algebra $F[G]$ as in the original discussion. The Hopf Galois structure for which $N=\rho(G)$ is the underlying group is called the \emph{classical} structure.
\end{example}

\begin{example}
It is clear that the action of $G$ on $\lambda(G)$ gives $G$-orbits equal to the conjugacy classes. When $G$ is not abelian (so that the $\rho(G)\neq\lambda(G)$) the structure for which $N=\lambda(G)$ is the underlying group is called the \emph{canonical non-classical} structure.
\end{example}

The theorem of Greither and Pareigis is the cornerstone of almost all of the work concerned with the enumeration, description, and application of Hopf-Galois structures on separable extensions of fields. In particular, via a theorem of Byott \cite[Proposition 1]{BYOTT1996}, it reveals a connection between the theory of Hopf-Galois structures and the theory of {\em left skew braces}, which is described in detail in the appendix to \cite{BYO2018}. This appendix contains an enumeration of the Hopf-Galois structures admitted by a quaternion extension $ L/F $ \cite[Table A.1]{BYO2018}. In section 2 below we compute the regular subgroups corresponding to these Hopf-Galois structures, and in section 3 we determine which of the Hopf algebras that appear are isomorphic as Hopf algebras. In section 4 we study the $ F $-algebra structure of these Hopf algebras: under the assumption that $ F $ has characteristic zero, we find explicit bases for each Hopf algebra, compute their Artin-Wedderburn decompositions, and identify which are isomorphic as $ F $-algebras.\\

The first named author acknowledges funding support from the Faculty of Natural Sciences at Keele University. We are grateful to Prof. Alan Koch for his comments on an early draft of this paper, and to the anonymous referee for improvements to the exposition and interpretation of our results. 

\section{Structures on the extension}

Let $L/F$ be a Galois extension of fields with Galois group $G$ isomorphic to the quaternion group of order $8$. Let $G$ have generators $\sigma$ and $\tau$, that is
\begin{equation*}
G=\langle\sigma,\tau|\sigma^4=\tau^4=1,\sigma^2=\tau^2, \sigma\tau=\tau\sigma^{-1}\rangle.
\end{equation*}

There are 5 isomorphism types of groups of order 8: the elementary abelian group $C_2\times C_2\times C_2$, $C_4\times C_2$, the cyclic group $C_8$, the dihedral group $D_4$ and the quaternion group $Q_8$. As mentioned in the introduction, \cite[Table A.1]{BYO2018} includes a count of the Hopf-Galois structures admitted by $ L/F $, which we reproduce in \Cref{table_count} below. The same count appears in work of Crespo and Salguero  \cite[Table 3]{CRE2017}, as an application of an algorithm written in the computational algebra system Magma which gives all Hopf Galois structures on separable field extensions of a given degree. 

\begin{table}[h!]
\caption{The number of Hopf-Galois structures on a quaternionic extension}
\centering
\begin{tabular}{c|c}
Type & Number of structures\\
\hline
$C_2\times C_2\times C_2$ & 2\\
$C_4\times C_2$ & 6\\
$C_8$ & 6\\
$Q_8$ & 2\\
$D_4$ & 6\\
\end{tabular}
\label{table_count}
\end{table}

We now determine the regular subgroups of $ \perm(G) $ corresponding to these Hopf-Galois structures. We start with the subgroups corresponding to the Hopf-Galois structures of type $C_2\times C_2\times C_2$.

\begin{lemma} \label{lemma_Est}
Let $s,t\in\{\sigma,\tau\} $ with $ s\neq t$ and let $E_{s,t}$ be generated by $\lambda(s)\rho(t),\lambda(s^2)$, and $\lambda(t)\rho(st)$. Then $E_{s,t}$ is a regular subgroup of $\perm(G)$ that is normalized by $\lambda(G)$ and isomorphic to $C_2\times C_2\times C_2$. The groups $E_{\sigma,\tau}$ and $ E_{\tau,\sigma} $ are distinct, and are the underlying groups of the 2 Hopf-Galois structures of type $C_2\times C_2\times C_2$ on $ L/F $. 
\end{lemma}
\begin{proof}
The elements of $E_{s,t}$ are 
\begin{equation*}
1,\,\,\lambda(s^2),\,\,\lambda(s)\rho(t),\,\,\lambda(s^{-1})\rho(t),\,\,\lambda(t)\rho(st),\,\,\lambda(t^{-1})\rho(st),\,\,\lambda(st)\rho(s),\,\,\lambda((st)^{-1})\rho(s).
\end{equation*}
All of the non-identity elements above have order 2, so $E_{s,t}$ is isomorphic to $C_2\times C_2\times C_2$. It is clear that $ E_{s,t} \subset \perm(G) $ and $E_{s,t}\cdot1_{G}=G$; hence $E_{s,t}$ is a regular subgroup of $\perm(G)$. To show that $E_{s,t}$ is normalized by $\lambda(G)$, it is sufficient to show that it is normalized by $ \lambda(s) $ and $ \lambda(t) $. Using the fact that $\lambda(G)$ and $\rho(G)$ commute inside $ \perm(G) $ we have for example
\begin{align*}
\,^{s} \lambda(s)\rho(t) & = \lambda(sss^{-1})\rho(t) = \lambda(s)\rho(t) \\
\,^{t} \lambda(s)\rho(t) & = \lambda(tst^{-1})\rho(t) = \lambda(s^{-1})\rho(t).
\end{align*}
Similar calculations apply to the other elements, and so $E_{s,t}$ is normalized by $\lambda(G)$. Finally, we have $E_{s,t}\neq E_{t,s}$ since $ \lambda(t)\rho(s) $ lies in $ E_{t,s} $ but not in $ E_{s,t} $. Referring to \Cref{table_count} we see that $ E_{\sigma,\tau} $ and $ E_{\tau,\sigma} $ are the underlying groups of the two Hopf Galois structures of type $C_2\times C_2\times C_2$ on $ L/F $.
\end{proof}

We now find the subgroups corresponding to the Hopf-Galois structures of type $C_4\times C_2$ using a similar technique.

\begin{lemma} \label{lemma_Ast}
Let $s,t\in\{\sigma,\tau,\sigma\tau\} $ with $ s\neq t$ and let $A_{s,t}$ be generated by the permutations $\lambda(s)$ and $\rho(t)$. Then $A_{s,t}$ is a regular subgroup of $\perm(G)$ that is normalized by $\lambda(G)$ and isomorphic to $C_4\times C_2$. The 6 choices of the pair $ s,t $ yield distinct groups, and these are the underlying groups of the 6 structures of type $C_4\times C_2$ on $ L/F $.
\end{lemma}
\begin{proof}
We have $\langle\rho(t),\lambda(s)\rangle\cong C_4\times C_2$ since $\rho(t)$ and $\lambda(s)$ are both of order 4, commute with each other, and share the same square.
It is clear that $ A_{s,t} \subset \perm(G) $ and that for $g,h\in G$ we have $\lambda(g)\rho(h)\cdot1_{G}=gh^{-1}$; hence $A_{s,t}$ is a regular subgroup of $\perm(G)$. The verification that it is normalized by $ \lambda(G) $ is very similar to the verification in \Cref{lemma_Est}, using the fact that $\rho(G)$ and $\lambda(G)$ commute inside $ \perm(G) $. To show that the six choices of the pair $ s,t $ yield distinct groups, note that for each such pair the group $ A_{s,t} $ is the only one that contains $ \lambda(s) $ and $ \rho(t) $. Hence, by \Cref{table_count}, the groups $A_{s,t}$ are the underlying groups of the 6 Hopf Galois structures of type $C_4\times C_2$.

%%PART OF OLD ARGUMENT HERE:
%consider for example $\lambda(s)$: we have $\lambda(s)\in A_{s,t}\cap A_{s,st}$ but $\lambda(s)\notin A_{t,s}\cup A_{t,st} \cup A_{st,s} \cup A_{st, t}$. Moreover $\rho(t)\in A_{s,t}$ but $\rho(t) \notin A_{s,st}$. Repeating this argument for $\lambda(t)$ and $\lambda(st)$ shows that each of the six choices of the pair $s,t$ yields a distinct group.
\end{proof}

The subgroups corresponding to the Hopf-Galois structures of type $ C_{8} $ cannot be described in terms of combinations of elements from $ \lambda(G) $ and $ \rho(G) $, since the order of any such element is at most $ 4 $. 

\begin{lemma} \label{lemma_Cst}
Let $s,t\in\{\sigma,\tau,\sigma\tau\} $ with $ s\neq t$ and let $C_{s,t}$ be generated by the permutation $\eta_{s,t} $ defined in cycle notation by
\begin{equation*}
\eta_{s,t} = (1\,\,s\,\,t\,\,(st)^{-1}\,\,s^2\,\,s^{-1}\,\,t^{-1}\,\,(st)).
\end{equation*}
Then $C_{s,t}$ is a regular subgroup of $\perm(G)$ that is normalized by $\lambda(G)$ and isomorphic to $C_8$. The 6 choices of the pair $ s,t $ yield distinct groups, and these are the underlying groups of the 6 structures of type $C_8 $ on $ L/F $.
\end{lemma}
\begin{proof}
It is clear that $C_{s,t}$ is a subgroup of $ \perm(G) $ isomorphic to $C_8$. Moreover, we have $C_{s,t}\cdot 1_{G}=G$ since $\eta_{s,t}^k\cdot1_{G}=1_{G}$ if and only if $k \equiv 0 \pmod{8}$. Thus $C_{s,t}$ is a regular subgroup of $\perm(G)$. To show that $C_{s,t}$ is normalized by $\lambda(G)$, it is sufficient to show that it is normalized by $ \lambda(s) $ and $ \lambda(t) $. We have
\begin{align*}
\lambda(s)&\eta_{s,t}\lambda(s^{-1})\\
&=(1\,\,s\,\,s^2\,\,s^{-1})(t\,\,st\,\,t^{-1}\,\,(st)^{-1})(1\,\,s\,\,t\,\,(st)^{-1}\,\,s^2\,\,s^{-1}\,\,t^{-1}\,\,st)(1\,\,s^{-1}\,\,s^2\,\,s)(t\,\,(st)^{-1}\,\,t^{-1}\,\,st)\\
&=(1\,\,(st)^{-1}\,\,t^{-1}\,\,s\,\,s^2\,\,st\,\,t\,\,s^{-1})\\
&=\eta_{s,t}^3,
\end{align*}
and similarly, $\lambda(t)\eta_{s,t}\lambda(t^{-1})=\eta_{s,t}$. Therefore $C_{s,t}$ is normalized by $\lambda(G) $. It may be verified that each of the 6 choices of the pair $s,t$ gives a permutation that differs from all powers of those of the other choices. Hence, by \Cref{table_count}, the groups $C_{s,t}$ are the underlying groups of the 6 Hopf Galois structures of type $C_8$.
\end{proof}

Having found the abelian underlying groups of the corresponding Hopf Galois structures on our extension $L/F$ we now find the structures of quaternionic type which we saw earlier.

\begin{lemma} 
$\rho(G)$ and $\lambda(G)$ are the underlying groups of the the two Hopf-Galois structures of type $Q_8$.
\end{lemma}
\begin{proof}
As $G$ is non-abelian, $\rho(G) $ and $ \lambda(G)$ are distinct regular subgroups of $ \perm(G) $ normalized by $ \lambda(G) $. By \Cref{table_count}, they are the underlying groups of the 2 Hopf-Galois structures of type $Q_{8}$. 
\end{proof}

Finally, the subgroups corresponding to the Hopf-Galois structures of type $D_4$, the dihedral group of order 8, have a similar description to the groups $ E_{s,t} $ and $ A_{s,t} $.

\begin{lemma} \label{lemma_Ds}
Let $s,t\in\{\sigma,\tau,\sigma\tau\}$ with $ s \neq t $. Let $D_{s,\lambda}$ be generated by $\lambda(s)$ and $\lambda(t)\rho(s)$, and let $D_{s,\rho}$ be generated by $\rho(s)$ and $\lambda(s)\rho(t)$. Then $D_{s,\lambda}$ and $D_{s,\rho}$ do not depend upon the choice of $t$, and are regular subgroups of $\perm(G)$ that are normalized by $\lambda(G)$ and isomorphic to $D_4$. The $ 3 $ choices of $ s $ yield $ 6 $ distinct groups, and these are the underlying groups of the Hopf Galois structures of type $D_4$ on $ L/F $.
\end{lemma}
\begin{proof} 
For a fixed choice of $ t $ the elements of $D_{s,\lambda}$ are
\begin{equation*}
1,\,\lambda(s),\,\lambda(s^2),\,\lambda(s^{-1}),\,\lambda(t)\rho(s),\,\lambda(st)\rho(s),\,\lambda(t^{-1})\rho(s), \,\lambda((st)^{-1})\rho(s).
\end{equation*}
We see immediately that using $ st $ in place of $ t $ yields the same group, that $ \lambda(s) $ has order $ 4 $, $ \lambda(t)\rho(s) $ has order $ 2 $, and that these elements anticommute. Therefore $ D_{s,\lambda} \cong D_{4} $. It is clear that $ D_{s,\lambda} \subset \perm(G) $ and that $D_{s,\lambda}\cdot1_{G}=G$; hence  $D_{s,\lambda}$ is a regular subgroup of $\perm(G)$. The verification that it is normalized by $ \lambda(G) $ is very similar to the verifications in \Cref{lemma_Est} and \Cref{lemma_Ast}, using the fact that $\rho(G)$ and $\lambda(G)$ commute inside $ \perm(G) $. Similarly, $ D_{s,\rho} $ is a regular subgroup of $ \perm(G) $ that is isomorphic to $ D_{4} $ and normalized by $ \lambda(G) $. To show that the $ 3 $ choices of $ s $ yield $ 6 $ distinct groups, note that for each $ s $ the group $ D_{s,\lambda} $ is the only one that contains $ \lambda(s) $ and that $ D_{s,\rho} $ is the only one that contains $ \rho(s) $. Hence, by \Cref{table_count}, the groups  $D_{s,\lambda}$ and $D_{s,\rho}$ are the underlying groups of the 6 Hopf-Galois structures of type $D_4$.
\end{proof}

\begin{remark}
For every regular subgroup $ N $ of $ \perm(G) $ corresponding to a Hopf-Galois structure on $ L/F $ we have $ \rho(\sigma^{2}) \in N $, and so $ Z(\rho(G)) \subseteq \rho(G) \cap N $. Clearly this is the case for $ N=\rho(G) $ and $ N=\lambda(G) $, and it is easy to verify that it holds for $ N = E_{s,t},  A_{s,t}, D_{s,\lambda}, $ and $ D_{s,\rho} $ (for all valid choices of $ s,t $) from the definitions of these groups. Finally, we can verify that it holds for the groups $ C_{s,t} $ (for all valid choices of $ s,t $) by computing $ \eta_{s,t}^{4} = \rho(\sigma^{2}) $ in these cases. 
\end{remark}

\section{Hopf algebra isomorphisms}

In this section we determine which of the Hopf algebras giving Hopf-Galois structures on $ L/F $ are isomorphic as $F$-Hopf algebras. In \cite[Theorem 2.2]{TARP2017} Koch, Kohl, Underwood and the second named author outline the following criterion for two Hopf algebras arising from the Greither-Pareigis correspondence to be isomorphic as Hopf algebras: let $N_1$ and $N_2$ be underlying groups of two Hopf Galois structures on $L/F$. Then $L[N_1]^G\cong L[N_2]^G$ as $F$-Hopf algebras if and only if there exists a $G$-equivariant isomorphism $f:N_1\xrightarrow{\sim}N_2$. In particular, no two Hopf algebras of different types may be isomorphic as $F$-Hopf algebras. 

%Suppose, for some $s\in\{\sigma,\tau,\sigma\tau\}$, that $\langle s\rangle=\stab_G(\eta)$ for some $\eta\in N_1$. Due to $G$-equivariance we require an isomorphism $f$ such that $s\cdot f(\eta)=f(s\cdot \eta)=f(\eta)$. Further, for some $t\neq s$, we must have $t\cdot f(\eta)=f(t\cdot \eta)\neq f(\eta)$.

We now determine which of our Hopf algebras are isomorphic. We consider the isomorphism classes of the underlying groups individually. We start with the elementary abelian groups.

\begin{lemma} \label{3.1}
The Hopf algebras giving the two Hopf-Galois structures of type $C_2\times C_2\times C_2$ are isomorphic to each other as Hopf algebras. That is, $L[E_{\sigma,\tau}]^G\cong L[E_{\tau,\sigma}]^G$ as Hopf algebras.  
\end{lemma}
\begin{proof}
Recall the definition of $ E_{s,t} $ from \Cref{lemma_Est}. The non-trivial $G$-orbits of $E_{s,t}$ are $\{\lambda(s)\rho(t),\lambda(s^{-1})\rho(t)\}$, $\{\lambda(t)\rho(st),\lambda(t^{-1})\rho(st)\}$ and $\{\lambda(st)\rho(s),\lambda((st)^{-1})\rho(s)\}$ with stabilisers $\langle s\rangle$, $\langle t \rangle$ and $\langle st \rangle$ respectively. The map $ f: E_{s,t}\rightarrow E_{t,s} $ defined by 
\begin{align*}
f: &\left\{\begin{array}{rcl}
\lambda(s)\rho(t) & \mapsto & \lambda(s)\rho((st)^{-1})\\
\lambda(s^2) & \mapsto & \lambda(s^2)\\
\lambda(t)\rho(st) & \mapsto & \lambda(t)\rho(s).
\end{array}\right.
\end{align*}
is a $ G $-equivariant isomorphism.
\end{proof}

Now we find that for the Hopf-Galois structures of type $C_4\times C_2$ the Hopf algebra isomorphism classes are determined by the choice of $s$.

\begin{lemma}\label{3.2}
Let $s, s^{\prime},t,t^{\prime}\in\{\sigma,\tau,\sigma\tau\}$ with $ s \neq t $ and $ s^{\prime} \neq t^{\prime} $. We have $L[A_{s,t}]^G\cong L[A_{s^{\prime},t^{\prime}}]^G$ if and only if $ s=s^{\prime} $. 
\end{lemma}
\begin{proof}
Recall the definition of $ A_{s,t} $ from \Cref{lemma_Ast}. The non-trivial $G$-orbits of $A_{s,t}$ (that is, those of length greater than one) are $\{\lambda(s),\lambda(s^{-1})\}$, $\{\lambda(s)\rho(t),\lambda(s^{-1})\rho(t)\}$ both with stabiliser $\langle s \rangle$. Therefore if $ s \neq s^{\prime} $ then there cannot be a $ G $-equivariant isomorphism between $ A_{s,t} $ and $ A_{s^{\prime},t^{\prime}} $ for any choices of $ t,t^{\prime} $. For fixed $ s $ and $ t,t^{\prime} $ satisfying $ s \neq t $ and $ s \neq t^{\prime} $ the map $ f: A_{s,t}\rightarrow A_{s,t^{\prime}} $ defined by 
\begin{align*}
f:&\left\{\begin{array}{rcl}
\lambda(s) & \mapsto & \lambda(s)\\
\rho(t) & \mapsto & \rho(t^{\prime}).
\end{array}\right.
\end{align*}
is a $G$-equivariant isomorphism:
\end{proof}

With a nearly identical argument we now give the result for Hopf-Galois structures of type $C_8$.

\begin{lemma}\label{3.3}
Let $s, s^{\prime},t,t^{\prime}\in\{\sigma,\tau,\sigma\tau\}$ with $ s \neq t $ and $ s^{\prime} \neq t^{\prime} $.  We have $L[C_{s,t}]^G\cong L[C_{s^{\prime},t^{\prime}}]^G$ as Hopf algebras if and only if $ t=t^{\prime} $. 
\end{lemma}
\begin{proof}
Recall the definition of $ C_{s,t} $ from \Cref{lemma_Cst}. The nontrivial $G$-orbits of $C_{s,t}$ are $\{\eta_{s,t},\eta_{s,t}^{3}\}$, $\{\eta_{s,t}^{2},\eta_{s,t}^{6}\}$ and $\{\eta_{s,t}^{5},\eta_{s,t}^{7}\}$, all with stabiliser $\langle t \rangle$. Therefore if $ t \neq t^{\prime} $ then there cannot be a $ G $-equivariant isomorphism between $ C_{s,t} $ and $ C_{s^{\prime},t^{\prime}} $ for any choices of $ s,s^{\prime} $. For fixed $ t $ and $ s,s^{\prime} $ satisfying $ s \neq t $ and $ s^{\prime} \neq t $ let $\eta_{s,t}$ and $\eta_{s^{\prime},t}$ be generators of $C_{s,t}$ and $C_{s^{\prime},t}$ respectively; then the map $ f: C_{s,t}\rightarrow C_{s^{\prime},t} $ defined by 
\begin{align*}
f:\,&\eta_{s,t}\mapsto\eta_{s^{\prime},t}.
\end{align*}
is a $G$-equivariant isomorphism. 
\end{proof}

The result for the Hopf-Galois structures of type $ Q_{8} $ is an instance of a well known result (see \cite[Example 2.4]{TARP2017}, for example). 

\begin{lemma}\label{3.4}
The Hopf algebras $L[\lambda(G)]^G$ and $L[\rho(G)]^G$ are not isomorphic as Hopf algebras.
\end{lemma}
\begin{proof}
The $G$-action on $\rho(G)$ is trivial since $\lambda(G)$ and $\rho(G)$ commute. However, the $G$-action on $\lambda(G)$ is conjugation so that the $G$-orbits are the conjugacy classes. Therefore no $G$-equivariant isomorphism can exist.
\end{proof}

Finally, we can give the result for the Hopf-Galois structures of type $ D_{4} $. 

\begin{lemma} \label{3.5}
The Hopf algebras $L[D_{s,\lambda}]^G$ and $L[D_{s,\rho}]^G$ are pairwise nonisomorphic as Hopf algebras.
\end{lemma}
\begin{proof}
Recall the definitions of $ D_{s,\lambda} $ and $ D_{s,\rho} $ from \Cref{lemma_Ds}. The non-trivial $G$-orbits of $D_{s,\lambda}$ are 
\begin{equation*}
\{\lambda(s),\lambda(s^{-1})\}, \{\lambda(t)\rho(s),\lambda(t^{-1})\rho(s)\}\text{, and }\{\lambda(st)\rho(s), \lambda((st)^{-1})\rho(s)\},
\end{equation*}
with stabilisers $\langle s\rangle$, $\langle t \rangle$, and $\langle st \rangle$ respectively. If $ s \neq s^{\prime} $ and $ f : D_{s,\lambda} \rightarrow D_{s^{\prime},\lambda} $ is a $ G $-equivariant bijection then by considering stabilisers we see that $ f(\lambda(s)) = \lambda(t^{\prime})\rho(s^{\prime}) $ for some $ t^{\prime} $. But $ \lambda(s) $ has order $ 4 $, whereas $ \lambda(t^{\prime})\rho(s^{\prime}) $ has order $ 2 $. Therefore $ f $ cannot be an isomorphism.
%PORTION OF OLD ARGUMENT HERE
%For a $G$-equivariant bijection between $D_{s,\lambda}$ and $D_{u,\lambda}$ we require an $\eta_s$ fixed by $s$ which must map to an $\eta_u$ not fixed by $u$. Then $\eta_s$ must be in the first orbit, thus of order $4$, whereas $\eta_u$ must not be in the first orbit, thus of order $2$. That is, any $G$-equivariant bijection between two distinct groups $D_{s,\lambda}$ must map an element of order 4 to an element of order 2 and so cannot be an isomorphism.

The non-trivial $G$-orbits of $D_{s,\rho}$ are
\begin{equation*}
\{\lambda(s)\rho(t),\lambda(s^{-1})\rho(t)\}\text{ and }\{\lambda(s)\rho(st),\lambda(s^{-1})\rho(st)\}
\end{equation*}
both with stabiliser $\langle s\rangle$. Therefore if $ s \neq s^{\prime} $ then there cannot be a $ G $-equivariant isomorphism between $ D_{s,\lambda} $ and $ D_{s^{\prime}, \lambda} $.

Finally, there cannot be a $ G $-equivariant isomorphism between $ D_{s,\lambda} $ and $ D_{s^{\prime}, \rho} $ for any $ s,s^{\prime} $, since these groups have different numbers of $ G $-orbits. 
\end{proof}

These results agree with the number of isomorphism classes of Hopf algebras for each type given in Table 1 of \cite{CRE2017}. It may also be worth noting that our results imply that the Hopf-Galois structures of abelian type occur in pairs, with each pair arising from two different actions of a single Hopf algebra, and that, by contrast, each Hopf-Galois structure of nonabelian type arises from the action of a distinct Hopf algebra. 

\section{F-algebra isomorphisms}

In this section we investigate the $ F $-algebra structure of the Hopf algebras giving Hopf-Galois structures on $ L/F $. We assume that the characteristic of $F$ is not $ 2 $: this ensures the Hopf algebras are separable, hence semisimple, so that each has an Artin-Wedderburn decomposition (see section 3C of \cite{CUR1990}). 

We fix some notation. Since $L/F$ is a quaternionic extension it has a unique biquadratic subextension $K/F$ corresponding to the unique order 2 subgroup $\langle\sigma^2\rangle$ of $G$, so that $\gal(K/F)=G/\langle\sigma^2\rangle$. Let $s,t\in\{\sigma,\tau,\sigma\tau\}$ with $ s \neq t $, and let $\alpha,\beta$ be elements of $K$ such that $ \alpha^{2}, \beta^{2} \in F $, $s(\alpha)=\alpha, t(\alpha)=-\alpha, s(\beta)=-\beta$ and $t(\beta)=\beta$; note that $K=F(\alpha,\beta)$. We also fix an algebraic closure $F^{\text{alg}}$ of $ F $, and let $ \Omega = \gal(F^{\text{alg}}/F) $. 

%If $ N $ is the underlying subgroup of a Hopf algebra $ H = L[N]^{G} $ giving a Hopf-Galois structure on $ L/F $ and $ \chi $ is a character of $ N $ then the idempotent in $ L[N] $ corresponding to $\chi$ is given by
%\begin{equation*}
%e_\chi=\frac{1}{|N|}\sum_{\eta\in N}\chi(\eta) \eta.
%\end{equation*}
%Noting that $\lambda(\sigma^2)=\rho(\sigma^2)$ we can see that in every underlying group $N$ found in section 2 we have $\lambda(\sigma^{2}) \in N\cap\lambda(G)$, and so the idempotent
%\begin{equation*}
%\mathfrak{e}=\frac{1}{2}\big(1-\lambda(\sigma^2)\big)
%\end{equation*}
%corresponds to a character of each of the underlying groups.

If $ N $ is abelian then $H= L[N]^{G} $ is a commutative separable $ F $-algebra, and hence, by \cite[\S 6.3]{WATERHOUSE}, corresponds to a finite $ \Omega $-set. Specifically, $ L[N]^{G} $ corresponds to the $ \Omega $-set $ \widehat{N}  = \mbox{Hom}(N,F^{\text{alg}}) $, where $ \Omega $ acts on $ N $ by factoring through $ G $, and on $ \widehat{N} $ by $ \left(^{\omega}\chi \right)[\eta] = \omega(\chi(^{\omega^{-1}} \eta)) $ for all $ \eta \in N $ (in fact, the action of $ \Omega $ on $ \widehat{N} $ factors through $ \gal{(L^{\prime}/K)} $ for some cyclotomic extension $ L^{\prime} $ of $ L $). To make this correspondence explicit, let $\chi_1, \dots, \chi_s\in \widehat{N}$ be a set of representatives for the $\Omega $ orbits of $\widehat{N}$, and for each $i\in\{1,\ldots, s\}$ let $ F_{i} $ be the fixed field of  $\stab_{\Omega}(\chi_i)$; then 
\[ H \cong \prod_{i=1}^{s} F_{i} \mbox{ as $ F $-algebras.} \]
A result of B\"oltje and Bley \cite[Lemma 2.2]{BLE1999} shows how one may construct an $ F $-basis of $ L[N]^G $ corresponding to this decomposition: we have $ L[N]^{G} =  F^{\text{alg}}[N]^{\Omega} $, and the group algebra $ F^{\text{alg}}[N] $ has a basis of mutually orthogonal idempotents, each corresponding to an element of $ \widehat{N} $. The action of $ \Omega $ on $ F^{\text{alg}}[N] $  permutes these idempotents, and by forming $ \Omega $-invariant linear combinations we obtain an $ F $-basis of $ L[N]^{G} $ corresponding to the decomposition above. 

If $ H=L[N]^G $ is a Hopf algebra whose underlying group $ N $ is isomorphic to $C_2\times C_2\times C_2$ then the  values of the characters of $ N $ lie in $ F $, so the action of $ \Omega $ on $ \widehat{N} $ factors through $ G $. Using this observation we have: 

\begin{lemma} 
Let $ E_{s,t} $ be defined as in \Cref{lemma_Est}. Then we have 
\[ L[E_{s,t}]^G\cong F^4\times K  \mbox{ as $ F $-algebras.} \]
\end{lemma}
\begin{proof}
 The dual group $\widehat{E}_{s,t}$ is generated by three characters:
\begin{align*}
&\chi_1:\left\{
\begin{array}{rcl}
\lambda(s)\rho(t)&\mapsto& -1\\
\lambda(s^2)&\mapsto& 1\\
\lambda(t)\rho(st)&\mapsto& 1\\
\end{array}\right.,\quad
\chi_2:\left\{
\begin{array}{rcl}
\lambda(s)\rho(t)&\mapsto& 1\\
\lambda(s^2)&\mapsto& -1\\
\lambda(t)\rho(st)&\mapsto& 1\\
\end{array}\right.,\quad
\chi_3:\left\{
\begin{array}{rcl}
\lambda(s)\rho(t)&\mapsto& 1\\
\lambda(s^2)&\mapsto& 1\\
\lambda(t)\rho(st)&\mapsto& -1\\
\end{array}\right..
\end{align*}
Let $ \chi_{0} $ denote the identity in $ \widehat{E}_{s,t} $, and recall the $G$-orbit structure of $E_{s,t}$ in \Cref{3.1}. It is easily verified that $^s\chi_2=\chi_2\chi_3$, $^t\chi_2=\chi_1\chi_2$ and $^{st}\chi_2=\chi_1\chi_2\chi_3$ and that $s$ and $t$ act trivially on $\chi_0$, $\chi_1$, $\chi_3$ and $\chi_1\chi_3$. Hence the orbits of $ G $ in $\widehat{E}_{s,t} $ are 
\[ \{\chi_0\},\quad \{\chi_1\},\quad \{\chi_3\}, \quad \{\chi_1\chi_3\}, \text{ and } \{\chi_2,\chi_1\chi_2,\chi_2\chi_3,\chi_1\chi_2\chi_3\}. \]
The orbit representatives $ \chi_0, \;  \chi_1, \; \chi_3 \; \chi_1\chi_3 $ all have stabilizer $ G $, and the orbit representative $ \chi_2 $ has stabiliser $ \langle s^2\rangle$. Therefore we have $L[E_{s,t}]^G\cong F^4\times K $, as claimed. 
\end{proof}

For the remaining structures whose underlying group $ N $ is abelian there may exist characters of $ N $ whose values do not lie in the field $F$. In these cases the action of $ \Omega $ on $ \widehat{N} $ depends upon the intersection of $ L $ with certain cyclotomic extensions of $ F $, and can be difficult to trace in detail. To overcome this problem we study the action of $ \Omega $ on the group algebra $ F^{\text{alg}}[N] $, as in \cite[Lemma 2.2]{BLE1999}. As discussed above, we have $ L[N]^{G} =  F^{\text{alg}}[N]^{\Omega} $, and the action of $ \Omega $ factors through $ \gal{(L^{\prime}/K)} $ for some cyclotomic extension $ L^{\prime} $ of $ L $. Thus, writing $ G^{\prime} = \gal{(L^{\prime}/L)} $, we have 
\[ L[N]^{G} = \left( L^{\prime}[N]^{G^{\prime}} \right)^{G}, \]
where the action of $ G^{\prime} $ on $ L^{\prime}[N] $ is only on the coefficients. In the following two lemmas we suppress the details of this first step of the descent (if any), and begin with a convenient $ L $-basis on $ L[N] $ on which it is easy to follow the action of $ G $. By forming $ G $-invariant linear combinations of these basis elements we obtain a basis of $ L[N]^{G} $ corresponding to its Artin-Wedderburn decomposition. Although working with bases in this way is rather cumbersome, it has the advantage of applying uniformly, whereas studying the orbits of $ \Omega $ in $ \widehat{N} $ can split into many cases, depending upon the roots of unity contained in $ L $. 

We continue with the Hopf algebras giving the structures of type $C_4\times C_2$.

\begin{lemma} 
Let $ A_{s,t} $ be defined as in \Cref{lemma_Ast}. Then we have
\[ L[A_{s,t}]^G\cong F^4\times F(\alpha,\iota)^d \mbox{ as $ F $-algebras,} \]
where $\iota\in F^{\text{alg}}$ is such that $\iota^2=-1$ and $d=2/[F(\alpha, \iota):F(\alpha)]$.
\end{lemma}
\begin{proof}
Let
\begin{align*}
b_{0}&=\frac{1}{8}\big(1+\lambda(s)+\lambda(s^2)+\lambda(s^{-1})+\rho(t)^{-1}+\lambda(s)^{-1}\rho(t)+\rho(t)+\lambda(s)\rho(t)\big),\\
b_{1}&=\frac{1}{8}\big(1-\lambda(s)+\lambda(s^2)-\lambda(s^{-1})-\rho(t)^{-1}+\lambda(s)^{-1}\rho(t)-\rho(t)+\lambda(s)\rho(t)\big),\\
b_{2}&=\frac{1}{8}\big(1+\lambda(s)+\lambda(s^2)+\lambda(s^{-1})-\rho(t)^{-1}-\lambda(s)^{-1}\rho(t)-\rho(t)-\lambda(s)\rho(t)\big),\\
b_{3}&=\frac{1}{8}\big(1-\lambda(s)+\lambda(s^2)-\lambda(s^{-1})+\rho(t)^{-1}-\lambda(s)^{-1}\rho(t)+\rho(t)-\lambda(s)\rho(t)\big), \\
b_{4}&=\frac{1}{4}\big(1-\lambda(s^2)+\lambda(s)^{-1}\rho(t)-\lambda(s)\rho(t)\big),\\
b_{5}&=\frac{1}{4}\big(1-\lambda(s^2)-\lambda(s)^{-1}\rho(t)+\lambda(s)\rho(t)\big),\\
b_{6}&=\frac{1}{4}\big(\lambda(s)-\lambda(s^{-1})-\rho(t)^{-1}+\rho(t)\big),\\
b_{7}&=\frac{1}{4}\big(-\lambda(s)+\lambda(s^{-1})-\rho(t)^{-1}+\rho(t)\big).
\end{align*}
It is easily verified that these $ 8 $ elements of $L[A_{s,t}]$ are linearly independent over $L$ and so form an $ L $-basis of $L[A_{s,t}]$. Recall from \Cref{3.2} that the non-trivial $G$-orbits of $A_{s,t}$, are $\{\lambda(s),\lambda(s^{-1})\}$, $\{\lambda(s)\rho(t),\lambda(s^{-1})\rho(t)\}$, both with stabiliser $\langle s \rangle$. From this we see that $ b_{0}, b_{1}, b_{2} $ and $ b_{3} $ are fixed by $ G $, that $ \,^{t}b_{4}=b_{5} $, and that $ \,^{t}b_{6}=b_{7} $. Therefore the following linear combinations of the above elements are all fixed by $G$, and in fact form a basis of $L[A_{s,t}]^G$ over $F$. 
\begin{align*}
a_{0}&=b_{0}, \\
a_{1}&=b_{1}\\
a_{2}&=b_{2},\\
a_{3}&=b_{3}, \\
a_{4,0}&=b_{4}+b_{5}=\frac{1}{2}\big(1-\lambda(s^2)\big)=\mathfrak{e},\\
a_{4,1}&=\alpha(b_{4}-b_{5})=-\alpha\mathfrak{e}\lambda(s)\rho(t),\\
a_{4,2}&=b_{6}+b_{7}=\mathfrak{e}\rho(t),\\
a_{4,3}&=\alpha(b_{6}-b_{7})=\alpha\mathfrak{e}\lambda(s).
\end{align*}

We have $ {a}_i{a}_j = \delta_{i,j} {a}_i $ for $ i,j=0,1,2,3 $ and $a_{4,k}a_i={a}_i{a}_{4,k}=0$ for all $i=0,1,2,3$ and $k=0,1,2,3$. Finally, we consider the multiplication table of the ${a}_{4,k}$.

\begin{center}
\begin{tabular}{c|cccc}
& $a_{4,0}$ & $a_{4,1}$ & $a_{4,2}$ & $a_{4,3}$\\
\hline
$a_{4,0}$ & $a_{4,0}$ & $a_{4,1}$ & $a_{4,2}$ & $a_{4,3}$\\
$a_{4,1}$ & $a_{4,1}$ & $\alpha^2 a_{4,0}$ & $a_{4,3}$ & $\alpha^2 a_{4,2}$\\
$a_{4,2}$ & $a_{4,2}$ & $a_{4,3}$ & $-a_{4,0}$ & $-a_{4,1}$\\
$a_{4,3}$ & $a_{4,3}$ & $\alpha^2 a_{4,2}$ & $-a_{4,1}$ & $-\alpha^2 a_{4,0}$\\
\end{tabular}
\end{center}
From the table it is clear that we have the claimed decomposition.
\end{proof}

We use a similar process for the Hopf algebras giving the Hopf-Galois structures of type $C_8$.

\begin{lemma} 
Let $ C_{s,t} $ be defined as in \Cref{lemma_Cst}. Then we have 
\[ L[C_{s,t}]^G\cong F^2\times F(\beta\iota)^{d_1}\times F(r\iota,\beta\iota)^{d_1d_2} \mbox{ as $ F $-algebras,} \]
where $r,\iota\in F^{\text{alg}} $ such that $r^2=2$, $\iota^2=-1$ and where $d_1=2/[F(\beta\iota):F]$ and $d_2=2/[F(r\iota,\beta\iota):F(\beta\iota)]$.
\end{lemma}
\begin{proof}
Let $\eta = \eta_{s,t}$ as defined in \Cref{lemma_Cst}, so that $ C_{s,t} = \langle \eta \rangle $, and let
\begin{align*}
b_{0}&=\frac{1}{8}\big(1+\eta+\eta^2+\eta^3+\eta^4+\eta^5+\eta^6+\eta^7\big),\\
b_{1}&=\frac{1}{8}\big(1-\eta+\eta^2-\eta^3+\eta^4-\eta^5+\eta^6-\eta^7\big), \\
b_{2}&=\frac{1}{4}\big(1-\eta^2+\eta^4-\eta^6\big),\\
b_{3}&=\frac{1}{4}\big(\eta-\eta^3+\eta^5-\eta^7\big),\\
b_{4}&=\frac{1}{2}\big(1-\eta^4\big),\\
b_{5}&=\frac{1}{2}\big(\eta^3-\eta^7\big), \\
b_{6}&=\frac{1}{2}\big(\eta^2-\eta^6\big),\\
b_{7}&=\frac{1}{2}\big(\eta-\eta^5\big).
\end{align*}

It is easily verified that these $ 8 $ elements of $L[C_{s,t}]$ are linearly independent over $ L $ and so form an $ L $-basis of $ L[C_{s,t}] $. Recall from \Cref{3.3} that the nontrivial $G$-orbits of $ C_{s,t} $ are $\{\eta,\eta^{3}\}$, $\{\eta^{2},\eta^{6}\}$ and $\{\eta^{5},\eta^{7}\}$, all with stabiliser $\langle t \rangle$. From this we see that $ b_{0}, b_{1}, b_{2} $ and $ b_{4} $ are fixed by $ G $, that $ \,^{s}b_{3} = -b_{3} $, $ \,^{s}b_{6}=-b_{6} $, and that $ \,^{s}b_{5}=b_{7} $. Therefore the following linear combinations of the above elements are all fixed by $G$, and in fact form a basis of $L[C_{s,t}]$ over $L$:

\begin{align*}
a_{0}&=b_{0},\\
a_{1}&=b_{1},\\
a_{2,0}&=b_{2},\\
a_{2,1}&=\beta b_{3}=\beta b_{2}\eta, \\
a_{3,0}&=b_{4}=\mathfrak{e}, \\
a_{3,1}&=\beta b_{6}=\beta\mathfrak{e}\eta^2,\\
a_{3,2}&=(b_{5}+b_{7}) = \mathfrak{e}(\eta^{3}+\eta), \\
a_{3,3}&=\beta(b_{5}-b_{7}) = \beta \mathfrak{e}(\eta^{3}-\eta). 
\end{align*}

We have $ {a}_i{a}_j = \delta_{i,j} {a}_i $ for $ i,j=0,1 $, ${a}_i{a}_{2,k}=0$ for $i=0,1$ and $k=0,1$, ${a}_i{a}_{3,k}=0$ for $i=0,1$ and $k=0,1,2,3$, and $ a_{2,k}a_{3,l}=0 $ for $ k=0,1 $ and $ l=0,1,2,3 $. Finally, we consider the multiplication tables of the $ a_{2,k} $ and the $a_{3,k}$.

\begin{center}
\begin{tabular}{c|cc}
& $a_{2,0}$ & $a_{2,1}$\\
\hline
$a_{2,0}$ & $a_{2,0}$ & $a_{2,1}$\\
$a_{2,1}$ & $a_{2,1}$ & $-\beta^2 a_{2,0}$\\
\end{tabular}
\end{center}

\begin{center}
\begin{tabular}{c|cccc}
& $a_{3,0}$ & $a_{3,1}$ & $a_{3,2}$ & $a_{3,3}$\\
\hline
$a_{3,0}$ & $a_{3,0}$ & $a_{3,1}$ & $a_{3,2}$ & $a_{3,3}$\\
$a_{3,1}$ & $a_{3,1}$ & $-\beta^{2} a_{3,0}$ & $a_{3,3}$ & $-\beta^{2} a_{3,2}$\\
$a_{3,2}$ & $a_{3,2}$ & $a_{3,3}$ & $-2 a_{3,0}$ & $-2a_{3,1}$\\
$a_{3,3}$ & $a_{3,3}$ & $-\beta^{2} a_{3,2}$ & $-2a_{3,1}$ & $2\beta^2 a_{3,0}$\\
\end{tabular}
\end{center}
From these tables it is clear that we have the claimed decomposition.
\end{proof}

Comparing these results with those obtained in section 3, we see that two Hopf algebras giving Hopf-Galois structures of the same abelian type on $ L/F $ are isomorphic as Hopf algebras if and only if they are isomorphic as $ F $-algebras. On the other hand, although Hopf algebras giving Hopf-Galois structures of different types are not isomorphic as Hopf, in certain situations it is possible that they are isomorphic as $ F $-algebras. For example: if $ \beta = \iota $ then $ L[E_{s,t}]^{G} \cong L[A_{s,t}]^{G} $ as $ F $-algebras. 

The remaining structures are of nonabelian type, and so we cannot employ the methods of \cite[\S 6.3]{WATERHOUSE} or \cite[Lemma 2.2]{BLE1999}. We emulate the same process using the character table in place of the dual group of our underlying group. We write down a convenient $ L $-basis of $L[N]$ and form $ G $-invariant linear combinations of these basis elements. We  find that certain quaternion algebras appear in the decompositions, and so we fix notation for these: for $ x,y \in F^{\times} $, let $ (x,y)_{F} $ denote the quaternion algebra with $ F $-basis $ 1,u,v,w $ satisfying the relations $ u^{2} = x $, $ v^{2} = y $, and $ uv = w = -vu $. In addition, let $ a= \alpha^{2} \in F^{\times} $, $ b=\beta^{2} \in F^{\times} $, where $ \alpha, \beta \in K $ are as defined at the beginning of this section.

We begin with the Hopf algebras giving the classical and canonical non-classical structures of type $Q_8$.

\begin{lemma} \label{lemma_classical_nonclassical_F_alg}
We have
\begin{equation*}
L[\rho(G)]^G\cong K[G] \cong F^4\times (-1,-1)_F \mbox{ as $ F $-algebras}
\end{equation*}
and 
\begin{equation*}
L[\lambda(G)]^G\cong F^4\times (-a,-b)_F  \mbox{ as $ F $-algebras.}
\end{equation*}
\end{lemma}
\begin{proof}
Let $ \mu \in \{\rho, \lambda \} $. The character table for $\mu(G)$ is
\begin{center}
\begin{tabular}{c|ccccc}
& 1 & $\{\mu(s^2)\}$ & $\{\mu(s), \mu(s^{-1})\}$ & $\{\mu(t), \mu(t^{-1})\}$ & $\{\mu(st), \mu((st)^{-1})\}$\\
\hline
$\chi_0$ & 1 & 1 & 1 & 1 & 1\\
$\chi_1$ & 1 & 1 & 1 & $-1$ & $-1$\\
$\chi_2$ & 1 & 1 & $-1$ & 1 & $-1$\\
$\chi_3$ & 1 & 1 & $-1$ & $-1$ & 1\\
$\psi$ & 2 & $-2$ & 0 & 0 & 0\\
\end{tabular}
\end{center}
%As previously noted the $G$-orbits of $\rho(G)$ are clearly trivial, hence $^g\chi_i=\chi_i$ for all $g\in G$ and $i\in\{0,1,2,3\}$. Since $\psi$ is not realisable over $\mathbb{Q}$  we extend $L$ as in the previous lemmas and fix fields to find the structure for $L[N]^G$.

First we consider the case $ \mu = \rho $, corresponding to the classical Hopf-Galois structure on $ L/F $. For $k=0,1,2,3$, let $ e_k $ be the orthogonal idempotent corresponding to the character $ \chi_{k} $. The idempotent corresponding to the 2-dimensional representation is
\begin{equation*}
e_{\psi}=\frac{1}{2}\Big(1-\rho(s^2)\Big)=\mathfrak{e}.
\end{equation*}
The following is a set of $ 8 $ linearly independent elements of $ L[\rho(G)] $, and each element is fixed by $ G $ since the action of $ G $ on $ \rho(G) $ is trivial. It is therefore a basis of $ L[\rho(G)]^{G} = F[\rho(G)] $ over $ F $: 
\begin{equation*}
\{e_0, e_1, e_2, e_3, \mathfrak{e}, \mathfrak{e}\rho(s), \mathfrak{e}\rho(t), \mathfrak{e}\rho(st)\}.
\end{equation*}
The $ e_{k} $ are orthogonal idempotents, and each is also orthogonal to every element of the set $\{\mathfrak{e}, \mathfrak{e}\rho(s), \mathfrak{e}\rho(t), \mathfrak{e}\rho(st)\}$. This set spans a $ 4 $-dimensional $ F $-algebra, which is isomorphic to the quaternion algebra $ (-1,-1)_{F} $ via the $ F$-algebra isomorphism defined by $\mathfrak{e}\rho(s)\mapsto\,u$, $\mathfrak{e}\rho(t)\mapsto\,v $. Therefore we have the claimed decomposition.

Now we consider the case $ \mu = \lambda $, corresponding to the canonical nonclassical Hopf-Galois structure on $ L/F $. As discussed in \Cref{3.4} the $G$-orbits of $\lambda(G)$ are the conjugacy classes. As above, for $k=0,1,2,3$ let $ e_k $ be the orthogonal idempotent corresponding to the character $ \chi_{k} $, and note that these are fixed by $ G $. The idempotent $\mathfrak{e}$, corresponding to the $2$-dimensional representation of $ \lambda(G) $, is also fixed by $ G $. Now consider the $L$-linearly independent set $\{\mathfrak{e}, \mathfrak{e}\lambda(s), \mathfrak{e}\lambda(t), \mathfrak{e}\lambda(st)\}$. An element of the $F$-algebra generated by this set is of the form
\begin{equation*}
x=a_0\mathfrak{e}+a_1\mathfrak{e}\lambda(s)+a_2\mathfrak{e}\lambda(t)+a_3\mathfrak{e}\lambda(st) \mbox{ with $a_k\in L$ for $ k=0,1,2,3 $.} 
\end{equation*}
The element $x$ is fixed by $G$ if and only if $a_1=a_1^\prime \alpha$, $a_2=a_2^\prime \beta$ and $a_3=a_3^\prime \alpha\beta$ for some $a_0,a_1^\prime, a_2^\prime, a_3^\prime \in F$. Thus the following set is an $ F $-basis of $ L[\lambda(G)]^{G} $:
\begin{equation*}
\{e_0, e_1, e_2, e_3, \mathfrak{e}, \alpha\mathfrak{e}\lambda(s), \beta\mathfrak{e}\lambda(t), \alpha\beta\mathfrak{e}\lambda(st)\}.
\end{equation*}
As above, the $ e_{k} $ are orthogonal to each other and to every element of the set $\{\mathfrak{e}, \alpha\mathfrak{e}\lambda(s), \beta\mathfrak{e}\lambda(t), \alpha\beta\mathfrak{e}\lambda(st)\}$. This set spans a $ 4 $-dimensional $ F $-algebra, which is isomorphic to the quaternion algebra $ (-a,-b)_{F} $ via the $ F$-algebra isomorphism defined by $\alpha\mathfrak{e}\lambda(s)\mapsto\, u, \beta\mathfrak{e}\lambda(t)\mapsto\, v$. Therefore we have the claimed decomposition.
\end{proof}

It may appear that the Hopf algebras giving the classical and canonical non-classical structures are not isomorphic as $F$-algebras. However, we have: 

\begin{lemma} \label{lemma_lambda_rho_quaternion_alg_iso} We have $(-a, -b)_F \cong (-1,-1)_F$ as $ F $-algebras.
\end{lemma}
\begin{proof}
By a result of Witt \cite[Theorem I.1.1]{JENSEN_YUI}, the fact that $ K = F(\alpha,\beta) $ embeds into a quaternionic extension of $ F $ implies that the quadratic form $ ax_{1}^{2} + bx_{2}^{2}+abx_{3}^{2} $ is equivalent to the quadratic form $ x_{1}^{2}+x_{2}^{2}+x_{3}^{2} $. These are the norm forms of the subspaces of pure quaternions of $ (-a, -b)_{F} $ and $ (-1,-1)_{F} $, respectively. Therefore these subspaces are isometric, and so (see \cite[III, Theorem 2.5]{LAM}) $ (-a, -b)_{F} \cong (-1,-1)_{F} $ as $ F $-algebras. 
\end{proof}

\begin{cor} We have $ L[\rho(G)]^G\cong L[\lambda(G)]^{G} \cong F^4\times (-1,-1)_F $ as $ F $-algebras. 
\end{cor}

In fact, this result follows from an unpublished theorem of Greither which states that if $ L/F $ is {\em any} Galois extension of fields then $ F[G] \cong L[\lambda(G)]^{G} $ as $ F $-algebras. See \cite[Theorem 5.2]{TARP2017_2} for more details. 

Finally, we have the Hopf algebras giving the structures of type $D_4$.

\begin{lemma} 
Let $ D_{s,\lambda} $ and $ D_{s,\rho} $ be defined as in \Cref{lemma_Ds}. Then we have 
\begin{equation*}
L[D_{s,\lambda}]^G\cong F^4\times (-a,b)_F \mbox{ as $ F $-algebras}
\end{equation*}
and
\begin{equation*}
L[D_{s,\rho}]^G\cong F^4\times (-1,a)_F \mbox{ as $ F $-algebras}.
\end{equation*}
\end{lemma}
\begin{proof}
The character table for $D_{s,\lambda}$ is the following:
\begin{center}
\begin{tabular}{c|ccccc}
& 1 & $\{\lambda(s^2)\}$ & $\{\lambda(s), \lambda(s^{-1})\}$ & $\{\lambda(t)\rho(s), \lambda(t^{-1})\rho(s)\}$ & $\{\lambda(st)\rho(s), \lambda((st)^{-1})\rho(s)\}$\\
\hline
$\chi_0$ & 1 & 1 & 1 & 1 & 1\\
$\chi_1$ & 1 & 1 & 1 & $-1$ & $-1$\\
$\chi_2$ & 1 & 1 & $-1$ & 1 & $-1$\\
$\chi_3$ & 1 & 1 & $-1$ & $-1$ & 1\\
$\psi$ & 2 & $-2$ & 0 & 0 & 0\\
\end{tabular}
\end{center}

As in the proof of \Cref{lemma_classical_nonclassical_F_alg}, for $k=0,1,2,3$ let $ e_k $ be the orthogonal idempotent corresponding to the character $ \chi_{k} $, and note that the idempotent corresponding to the 2-dimensional representation is $ \mathfrak{e} $. Recall from \Cref{3.5} that the non-trivial $G$-orbits of $D_{s,\lambda}$ are 
\begin{equation*}
\{\lambda(s),\lambda(s^{-1})\}, \{\lambda(t)\rho(s),\lambda(t^{-1})\rho(s)\}\text{, and }\{\lambda(st)\rho(s), \lambda((st)^{-1})\rho(s)\},
\end{equation*}
with stabilisers $\langle s\rangle$, $\langle t \rangle$, and $\langle st \rangle$ respectively. Hence each $ e_{k} $ is fixed by $ G $. Now consider the $L$-linearly independent set $\{\mathfrak{e}, \mathfrak{e}\lambda(s), \mathfrak{e}\lambda(t)\rho(s), \mathfrak{e}\lambda(st)\rho(s)\}$. An element of the $F$-algebra generated by these elements is of the form
\begin{equation*}
x=a_0 \mathfrak{e}+a_1 \mathfrak{e}\lambda(s)+a_2 \mathfrak{e}\lambda(t)\rho(s)+a_3 \mathfrak{e}\lambda(st)\rho(s) \mbox{ with $a_k\in L$ for $ k=0,1,2,3 $.} 
\end{equation*}
The element $ x $ is fixed by $ G $ if and only if $a_1=a_1^\prime \alpha$, $a_2=a_2^\prime \beta$ and $a_3=a_3^\prime \alpha\beta$ for some $a_0, a_1^\prime, a_2^\prime, a_3^\prime\in F$. The set 
\begin{equation*}
\{e_0, e_1, e_2, e_3, \mathfrak{e}, \alpha \mathfrak{e}\lambda(s), \beta\mathfrak{e}\lambda(t)\rho(s), \alpha\beta\mathfrak{e}\lambda(st)\rho(s)\}
\end{equation*}
This set is therefore an $ F $-basis of $ L[D_{s,\lambda}]^{G} $. The $ e_{k} $ are orthogonal to each other and to every element of the set $\{ \mathfrak{e}, \alpha \mathfrak{e}\lambda(s), \beta\mathfrak{e}\lambda(t)\rho(s), \alpha\beta\mathfrak{e}\lambda(st)\rho(s) \}$. This set spans a $ 4 $-dimensional $ F $-algebra,
which is isomorphic to the quaternion algebra $ (-a,b)_{F} $ via the $ F$-algebra isomorphism defined by $\alpha\mathfrak{e}\lambda(s)\mapsto\, u, \beta\mathfrak{e}\lambda(t)\rho(s)\mapsto\, v$. Therefore we have the claimed decomposition. 

%\begin{center}
%\begin{tabular}{c|cccc}
%& $\mathfrak{e}$ & $\mathfrak{e}_\alpha$ & $\mathfrak{e}_\beta$ & $\mathfrak{e}_{\alpha\beta}$\\
%\hline
%$\mathfrak{e}$ & $\mathfrak{e}$ & $\mathfrak{e}_\alpha$ & $\mathfrak{e}_\beta$ & $\mathfrak{e}_{\alpha\beta}$\\
%$\mathfrak{e}_\alpha$ & $\mathfrak{e}_\alpha$ & $-\alpha^2\mathfrak{e}$ & $\mathfrak{e}_{\alpha\beta}$ & $-\alpha^2\mathfrak{e}_\beta$\\
%$\mathfrak{e}_\beta$ & $\mathfrak{e}_\beta$ & $-\mathfrak{e}_{\alpha\beta}$ & $\beta^2\mathfrak{e}$ & $-\beta^2\mathfrak{e}_\alpha$\\
%$\mathfrak{e}_{\alpha\beta}$ & $\mathfrak{e}_{\alpha\beta}$ & $\alpha^2\mathfrak{e}_\beta$ & $\beta^2\mathfrak{e}_\alpha$ & $(\alpha\beta)^2\mathfrak{e}$\\
%\end{tabular}
%\end{center}
%From this table we see that $L[D_{s,\lambda}]^G \cong F^{4} \times (-a,b)_{F} $ as $ F $-algebras.

We determine the structure of  $L[D_{s,\rho}]^{G}$ by essentially the same method, and so we omit some of the details. In notation analogous to that employed above, we find that the set
\begin{equation*}
\{e_0,e_1,e_2,e_3,\mathfrak{e},\mathfrak{e}\rho(s),\alpha\mathfrak{e}\lambda(s)\rho(t),\alpha\mathfrak{e}\lambda(s)\rho(st)\}
\end{equation*}
is an $ F $-basis of $ L[D_{s,\lambda}]^{G} $. The final four elements span a $ 4 $-dimensional $ F $-algebra, which is isomorphic to the quaternion algebra $ (-1,a)_{F} $ via the $ F$-algebra isomorphism defined by $\mathfrak{e}\rho(s)\mapsto\, u, \alpha\mathfrak{e}\lambda(s)\rho(t)\mapsto\, v$. Therefore we have the claimed decomposition.

\end{proof}

As in the case of the Hopf algebras giving the Hopf-Galois structures of $ Q_{8} $ type, some of the quaternion algebras appearing in the decompositions above are isomorphic:

\begin{lemma}
We have $ (-a,b)_F \cong (-1,a)_F $ as $ F $-algebras. 
\end{lemma}
\begin{proof}
Write $ [-a, -b], [-1, a] $ for the classes of $ (-a,b)_F,  (-1,a)_F $ in the Brauer group $ \mathrm{Br}(F) $. It is sufficient to show that $ [-a, -b]= [-1, a] $. We refer to \cite[Chapters III and IV]{LAM} for properties of quaternion algebras over $ F $ and their classes in $ \mathrm{Br}(F) $. Using the result of \Cref{lemma_lambda_rho_quaternion_alg_iso} we have $ [-a, -b] = [-1,-1] $, and so in $ \mathrm{Br}(F) $ we have
\begin{eqnarray*}
[-a, b] [-a, -b] & = & [-a, -b^{2}] \mbox{ by \cite[III, Theorem 2.11]{LAM} } \\
& = & [-a,-1] \mbox{ by \cite[III, Proposition 1.1]{LAM}}\\
& = & [-1,-a] \\
& = & [-1,a][-1,-1] \mbox{ by \cite[III, Theorem 2.11]{LAM} } \\
& = & [-1,a][-a,-b]. 
\end{eqnarray*}
Cancelling $ [-a,-b] $, we obtain $ [-a,b] = [-1,a] = [a,-1] $, as claimed. Therefore  $ (-a,b)_F \cong (-1,a)_F $ as $ F $-algebras. 
\end{proof}

\begin{cor} We have 
\begin{equation*}
L[D_{s,\rho}]^G\cong L[D_{s,\lambda}]^G\cong F^4\times (-1,a)_F \mbox{ as $ F $-algebras.}
\end{equation*}
\end{cor}

In order to better understand the $ F $-algebra structure of the Hopf algebras $ L[D_{s,\rho}]^G $, we investigate the relationships between $ (-1,a)_{F}, (-1,b)_{F} $ and $ (-1,ab)_{F} $. 

\begin{lemma} \label{lemma_qalg_trivial_vs_iso}
Let $ x,y \in \{ a,b, ab \} $ with $ x \neq y $. Then we have $ (-1,x)_{F} \cong (-1,xy)_{F} $ as $ F $-algebras if and only if $ (-1,y)_{F} \cong M_{2}(F) $ as $ F $-algebras. 
\end{lemma}
\begin{proof}
In $ \mathrm{Br}(F) $ we have $ [-1,xy]=[-1,x][-1,y] $, so $ [-1,x]=[-1,xy] $ if and only if $ [-1,y] = [-1,1] $. That is, $ (-1,x)_{F} \cong (-1,xy)_{F} $ as $ F $-algebras if and only if   $ (-1,y)_{F} \cong (-1,1)_{F} \cong M_{2}(F) $ as $ F $-algebras. 
\end{proof}

\Cref{lemma_qalg_trivial_vs_iso} suggests three scenarios for the quaternion algebras $ (-1,a)_{F}, (-1,b)_{F} $ and $ (-1,ab)_{F} $: all three are isomorphic to matrix rings, exactly one is isomorphic to a matrix ring and the other two are isomorphic to the same division algebra, or each is isomorphic to a distinct division algebra. We conclude with examples illustrating that each of these three cases does occur. 

\begin{example}
Suppose that $ -1 $ is a square in $ F $. Then for $ x \in \{ a,b,ab \} $ we have that $ -1 $ occurs as the norm of an element of the field $ F(x) $, and so $ (-1,x)_{F} \cong (-1,1)_{F} \cong M_{2}(F) $ \cite[Proposition I.1.6]{JENSEN_YUI}. Therefore in this case we have 
\begin{equation*}
L[D_{s,\rho}]^G \cong L[D_{t,\rho}]^G \cong L[D_{st,\rho}]^G \cong F^{4} \times M_{2}(F) 
\end{equation*}
as $ F $-algebras. 
\end{example}

\begin{example}
Let $ F=\Q $, $ \alpha = \sqrt{11} $, $\beta = \sqrt{2} $. Then by \cite{FUJISAKI} $ K=\Q(\alpha, \beta) $ can be embedded in a quaternionic extension $ L $ of $ \Q $.  In this case we have $ (-1,b)_{\Q} \cong (-1,1)_{\Q} \cong M_{2}(\Q) $ a $ \Q $-algebras since $ 2 $ is the norm of the element $ 1+i \in \Q(i) $, and so by \Cref{lemma_qalg_trivial_vs_iso} we have $ (-1,a)_{\Q} \cong (-1,ab)_{\Q} $ as $ \Q $-algebras. However, $ (-1,a)_{\Q} \not \cong M_{2}(\Q) $, since no element of $ \Q(i) $ has norm $ 11 $. Therefore in this case we have $ L[D_{t,\rho}]^G \cong \Q^{4} \times M_{2}(\Q) $ and
\begin{equation*}
L[D_{s,\rho}]^G  \cong L[D_{st,\rho}]^G \cong F^{4} \times (-1,a) \not \cong F^{4} \times M_{2}(\Q) 
\end{equation*}
as $ \Q $-algebras. 
\end{example}

\begin{example}
Let $ F = \Q $, $ \alpha = \sqrt{11} $, $ \beta = \sqrt{6} $. Then by \cite[Example 4.4]{VAUGHAN} $ K =  \Q(\alpha,\beta) $ can be embedded in a quaternionic extension $ L $ of $ \Q $.  In this case none of $ (-1,a)_{\Q}, (-1,b)_{\Q}, (-1,ab)_{\Q} $ is isomorphic to $ M_{2}(\Q) $ as a $ \Q $-algebra, since none of $ 6,11,66 $ occurs as the norm of an element of $ \Q(i) $.  Therefore by \Cref{lemma_qalg_trivial_vs_iso} these quaternion algebras are all nonisomorphic as $ \Q $-algebras, and so we have
\begin{equation*}
L[D_{s,\rho}]^G  \not \cong L[D_{t,\rho}]^G \not \cong L[D_{st,\rho}]^G
\end{equation*}
as $ \Q $-algebras.  
\end{example}

\end{document}